\documentclass[11pt]{article}

\usepackage{amsmath, amsthm, amssymb}
\usepackage{hyperref}
\usepackage{booktabs}
\usepackage{geometry}
\usepackage{makecell}
\newtheorem{theorem}{Theorem}
\newtheorem{lemma}[theorem]{Lemma}
\newtheorem{corollary}[theorem]{Corollary}

\title{On the Structure of 3D Queen Domination}
\author{
  Mahesh Ramani\\
  Independent\\
  1105 Timber Oaks Rd\\
  Edison, New Jersey\\
  \texttt{mahesh.ramani.iyer@gmail.com}
}
\date{}

\begin{document}

\maketitle

\begin{abstract}
We study the domination number $\gamma(Q^3_n)$ of the three-dimensional
$n\times n\times n$ queen graph. The main result is a stratified theorem
computing, for each position type---corner, edge, face, or interior---the
number of inner-core vertices dominated by a queen, and showing in particular
that interior placements dominate strictly more core cells than boundary
placements. This yields a symmetry-reduction principle via the
octahedral group and complements the standard counting lower bound and layered
upper bound, giving $\gamma(Q^3_n)=\Theta(n^2)$. We also certify exact values
for $n\le 6$ via integer linear programming and independent verification.
\end{abstract}
\section{Introduction}

The queen's domination problem originated in the mid-19th century, formally proposed by C. F. de Jaenisch in 1862 as an extension of the non-attacking $n$-queens puzzle introduced by Max Bezzel in 1848. It asks for the minimum number of queens needed
to dominate an $n \times n$ chessboard: every cell must either contain a queen
or be attacked by one. A detailed account of the two-dimensional problem,
including the asymptotic result $\gamma(Q^2_n) \leq \tfrac{69}{133}n + O(1)$
and values of $\gamma(Q^2_n)$ for $n \leq 120$, is given in
\cite{osterCardWeakley}; see also \cite{weakley2022, finozhenok2007} for further
results. Chessboard domination problems more broadly are surveyed in
\cite{cockayne1990}.

The three-dimensional analogue is defined as follows. Let $Q^3_n$ denote the
graph whose vertex set is $[n]^3 := \{0,1,\ldots,n-1\}^3$ and in which two
distinct vertices are adjacent if one can be reached from the other by a 3D queen
move. A \emph{3D queen move} proceeds any number of steps along one of the 13
undirected line families through a cell:
\begin{itemize}
  \item 3 axis-parallel directions (along $x$, $y$, or $z$);
  \item 6 face-diagonal directions ($(\pm1,\pm1,0)$, $(\pm1,0,\pm1)$,
        $(0,\pm1,\pm1)$ up to sign);
  \item 4 space-diagonal directions ($(\pm1,\pm1,\pm1)$ up to sign).
\end{itemize}
The closed neighbourhood of $v$ is $N[v] = \{v\} \cup \{w : w \sim v\}$.
A set $S \subseteq [n]^3$ is a \emph{dominating set} if
$\bigcup_{v \in S} N[v] = [n]^3$, and $\gamma(Q^3_n)$ denotes its minimum
cardinality. The maximum degree of any vertex is $13(n-1)$, so the maximum
closed-neighbourhood size is $13n-12$.

The higher-dimensional queen problem was introduced by Barr and Rao
\cite{barrrao2007}, who established the first lower bound for domination in
$d$-dimensional queen graphs and showed that $n$ queens do not always suffice to
dominate an $n^d$-cell board when $d \geq 3$.

Section~\ref{sec:main} proves Theorem~\ref{thm:orbit}, which computes the
number of inner core cells covered by a queen at each position type, and
Corollary~\ref{cor:fundomain}, which gives a lossless symmetry reduction to the
fundamental domain of the board's symmetry group. Section~\ref{sec:bounds}
establishes asymptotic lower and upper bounds. Section~\ref{sec:computation}
describes how exact values for small $n$ are certified via integer linear
programming, including the symmetry reduction and the decomposed infeasibility
argument.

\section{Core Coverage by Position Type}
\label{sec:main}

Fix $n \geq 4$. Let $m = n-2$ and $C = \{1,\ldots,n-2\}^3$ be the
\emph{inner core} of $[n]^3$. The inner core is a useful region for measuring
queen efficiency: it is the set of non-boundary vertices of $Q^3_n$, and therefore
the region where all 13 line families extend in both directions.
A queen placed outside $C$ has one or more line families truncated or suppressed
entirely by proximity to the boundary. Making this precise requires computing,
for each position type, the number of core cells dominated. For $q \in[n]^3$
define $\kappa(q) := |N[q] \cap C|$.

The following lemma isolates the minimization used in case (iii).

\begin{lemma}
\label{lem:parity}
Let $M \geq 1$ be an integer and let $f(a,c) = |a-c| + |a+c-M|$ for integers
$0 \leq a, c \leq M$. Then $f(a,c) \equiv M \pmod{2}$ for all $a, c$. Hence
$f(a,c) \geq M \bmod 2$. If $M$ is even, equality holds if and only if
$a = c = M/2$. If $M$ is odd, equality holds exactly at
\[
(a,c) \in \left\{\left(\frac{M-1}{2},\frac{M-1}{2}\right),
\left(\frac{M-1}{2},\frac{M+1}{2}\right),
\left(\frac{M+1}{2},\frac{M-1}{2}\right),
\left(\frac{M+1}{2},\frac{M+1}{2}\right)\right\}.
\]
\end{lemma}

\begin{proof}
Since $|a-c| = a+c-2\min(a,c)$, we have $|a-c| \equiv a+c \pmod{2}$. Likewise,
$|a+c-M| \equiv a+c+M \pmod{2}$. Therefore
$f(a,c) \equiv 2(a+c)+M \equiv M \pmod{2}$, and since $f \geq 0$ the bound
$f \geq M \bmod 2$ follows.

If $M$ is even: $f = 0$ requires $a = c$ and $a+c = M$ simultaneously, forcing
$a = c = M/2$, which is a unique integer solution.

If $M$ is odd: $f = 1$ is achieved in two families. Either $a = c$ with
$|2a-M| = 1$ (giving $a = c = (M \pm 1)/2$), or $a+c = M$ with $|a-c| = 1$
(giving $\{a,c\} = \{(M \pm 1)/2\}$). These give the four stated solutions.
\end{proof}

\begin{theorem}
\label{thm:orbit}
Let $n \geq 4$ and $m = n-2$.
\begin{enumerate}
  \item[\textnormal{(i)}] If $q$ is a corner vertex, then $\kappa(q) = m$.
  \item[\textnormal{(ii)}] If $q$ lies on an edge but not at a corner, then
        $\kappa(q) = 2m-1$.
  \item[\textnormal{(iii)}] If $q$ lies on a face but not on any edge, then
        \[
          \kappa(q) \;\leq\; 5m - 4 - \varepsilon_m,
          \qquad \varepsilon_m = (m-1)\bmod 2,
        \]
        with equality when the face centre is a lattice point, which occurs
        precisely when $m$ is odd.
  \item[\textnormal{(iv)}] If $q \in C$, then $\kappa(q) \leq 13m-12$, with
        equality when $q$ is at the centre of the board.
\end{enumerate}
In particular, $\max_{q\in\partial[n]^3}\kappa(q) < \max_{q\in C}\kappa(q)$
for all $n \geq 4$.
\end{theorem}

\begin{proof}
Since the 13 queen-line directions through any cell are mutually non-parallel,
two distinct queen lines through the same cell intersect only at that cell.

\smallskip\noindent\textbf{(iv).}
For $q \in C$, each of the 13 queen directions determines a segment; its
intersection with $C$ contains at most $m$ cells (including $q$). Since the
segments share only $q$,
\[
  \kappa(q) \;\leq\; 13(m-1)+1 \;=\; 13m-12.
\]
A cell at the centre of $C$ attains all 13 full segments, so equality holds.

\smallskip\noindent\textbf{(i).}
At $q = (0,0,0)$, a direction $(d_x,d_y,d_z)$ reaches $C = \{1,\ldots,m\}^3$
only if $d_x = d_y = d_z = 1$. That space diagonal visits exactly $m$ core
cells, so $\kappa(q) = m$.

\smallskip\noindent\textbf{(ii).}
At $q = (0,0,z)$ with $1 \leq z \leq m$, the directions reaching $C$ must have
$d_x = d_y = 1$. These are $(1,1,0)$, $(1,1,1)$, $(1,1,-1)$, with
core-segment lengths $m$, $m-z$, $z-1$ respectively. Thus
\[
  \kappa(q) \;=\; m + (m-z) + (z-1) \;=\; 2m-1.
\]

\smallskip\noindent\textbf{(iii).}
At $q = (0,y,z)$ with  $1 \leq y,z \leq m$, set $a = y-1$, $b = m-y$,
$c = z-1$, $d = m-z$, $M = m-1$, so $a+b = c+d = M$. All nine directions with
$d_x = 1$ reach $C$; their core-segment lengths are
\[
  m,\quad a,\quad b,\quad c,\quad d,\quad
  \min(a,c),\quad \min(a,d),\quad \min(b,c),\quad \min(b,d).
\]
Using $b = M-a$, $d = M-c$, and $\min(u,v) = \tfrac{1}{2}(u+v-|u-v|)$,
\[
  \min(a,c)+\min(a,d)+\min(b,c)+\min(b,d) \;=\; 2M - |a-c| - |a+c-M|.
\]
Therefore
\[
  \kappa(q) \;=\; 5m-4 - f(a,c), \quad f(a,c) = |a-c|+|a+c-M|.
\]
By Lemma~\ref{lem:parity}, $f(a,c) \geq M \bmod 2 = (m-1) \bmod 2 =
\varepsilon_m$, giving $\kappa(q) \leq 5m-4-\varepsilon_m$. When $M$ is even
(i.e., $m$ is odd), Lemma~\ref{lem:parity} gives a unique minimiser $a = c =
M/2$, corresponding to $y = z = (m+1)/2$, the unique lattice centre of the
face. When $M$ is odd ($m$ even), the face has no lattice centre and the
minimum $\varepsilon_m = 1$ is attained at the four lattice points closest to
the geometric centre.

\smallskip\noindent\textbf{Comparison.}
For $m \geq 2$, the boundary coverage is at most $5m - 4 - \varepsilon_m$. By (iv), the interior maximum is $13m - 12$ when $m$ is odd. When $m$ is even, evaluating any of the eight central-most cells in $C$ yields a coverage of exactly $13m - 18$. In either case, $\max_{q\in C}\kappa(q) \geq 13m - 18$. We thus have the separation:
\[
  m \;<\; 2m-1 \;\leq\; 5m-5 \;\leq\; 5m-4-\varepsilon_m \;<\; 13m-18 \;\leq\; \max_{q\in C}\kappa(q),
\]
so $\max_{q\in\partial[n]^3}\kappa(q) < \max_{q\in C}\kappa(q)$.
\end{proof}

\begin{corollary}
\label{cor:fundomain}
Let $O_h$ be the group of order $48$ acting on $[n]^3$ by coordinate
permutations and sign reflections, and let
\[
  F(n) = \bigl\{(x,y,z): 0 \leq x \leq y \leq z \leq (n-1)/2\bigr\}
\]
be a fundamental domain. Every dominating set of $Q^3_n$ has an
$O_h$-equivalent representative with a queen in $F(n)$. Theorem~\ref{thm:orbit}
shows that boundary placements are never better than interior placements for
core coverage. Furthermore, the symmetry reduction ensures that equivalent
configurations are not counted multiple times.
\end{corollary}

\begin{proof}
Any dominating set $S$ may be replaced by $\sigma(S)$ for any $\sigma \in O_h$
without changing its cardinality or the domination property, since $O_h$ acts by
graph automorphisms of $Q^3_n$. Among all $O_h$-images of $S$, choose the
lexicographically smallest; this representative has its first queen in $F(n)$
by definition of the fundamental domain.
\end{proof}

\subsection*{Example: $n = 5$}

Take $n = 5$, so $m = 3$ and $C = \{1,2,3\}^3$. The four position types and
their core-coverage values are as follows.
\begin{itemize}
  \item \emph{Corner}, e.g.\ $q = (0,0,0)$: the only direction entering $C$ is
        $(1,1,1)$, which visits $(1,1,1)$, $(2,2,2)$, $(3,3,3)$. Thus
        $\kappa(q) = 3 = m$.

  \item \emph{Edge (non-corner)}, e.g.\ $q = (0,0,2)$: the three qualifying
        directions $(1,1,0)$, $(1,1,1)$, $(1,1,-1)$ contribute $3$, $1$, $1$
        core cells respectively, for a total of $\kappa(q) = 5 = 2m-1$.

  \item \emph{Face (non-edge)}, e.g.\ $q = (0,2,2)$: here $a = b = c = d = 1$
        and $M = 2$. Since $M$ is even, Lemma~\ref{lem:parity} gives the unique
        minimiser $a = c = 1$, attained at $q$ itself (the face centre). The
        nine directions contribute $3,1,1,1,1,1,1,1,1$ core cells, giving
        $\kappa(q) = 11 = 5m - 4 - \varepsilon_3 = 11$.

  \item \emph{Interior}, e.g.\ $q = (2,2,2)$: this is the unique centre of the
        board. All 13 directions run a full segment through $C$, each
        contributing $m = 3$ cells, so $\kappa(q) = 13(3)-12 = 27 = |C|$.
\end{itemize}

The strict inequalities $3 < 5 < 11 < 27$ instantiate the comparison in
Theorem~\ref{thm:orbit}. The centre $(2,2,2)$ is the only cell that dominates
all 27 cells of the core, and it alone suffices for $n = 3$ (where it is the
unique cell).

\section{Bounds}
\label{sec:bounds}

\begin{theorem}[{\cite{barrrao2007}}]
\label{thm:volume}
\[
  \gamma(Q^3_n) \;\geq\; \left\lceil \frac{n^3}{13n-12} \right\rceil.
\]
\end{theorem}

\begin{proof}
Because no cell can have a closed neighbourhood larger than $13n-12$, any
dominating set has cardinality at least $\lceil n^3/(13n-12)\rceil$.
\end{proof}

\begin{theorem}
\label{thm:projection}
$\gamma(Q^3_n) \geq \gamma(Q^2_n)$.
\end{theorem}

\begin{proof}
Let $S$ be a dominating set of $Q^3_n$ and let
$P(S) = \{(x,y) : (x,y,z) \in S \text{ for some } z\}$. For any $(x,y) \in[n]^2$ and any height $z$, the cell $(x,y,z)$ is dominated by some
$(a,b,c) \in S$. The difference $(a-x,b-y,c-z)$ is a valid 3D queen move, so
$(a-x,b-y)$ is a valid 2D queen move (or zero), and $(a,b) \in P(S)$ dominates
$(x,y)$ in $Q^2_n$. Hence $P(S)$ is a dominating set of $Q^2_n$, and
$|S| \geq |P(S)| \geq \gamma(Q^2_n)$.
\end{proof}

\begin{theorem}
\label{thm:lifting}
$\gamma(Q^3_n) \leq n\cdot\gamma(Q^2_n) \leq \tfrac{69}{133}\,n^2 + O(n)$.
\end{theorem}

\begin{proof}
Placing a minimum 2D dominating set in each layer $z = 0,\ldots,n-1$ dominates
every cell of $Q^3_n$, giving $\gamma(Q^3_n) \leq n\cdot\gamma(Q^2_n)$.
The bound $\gamma(Q^2_n) \leq \tfrac{69}{133}\,n + O(1)$ is due to
\"{O}sterg\r{a}rd and Weakley \cite{osterCardWeakley}.
\end{proof}

Theorems~\ref{thm:volume}--\ref{thm:lifting} give
\[
  \left\lceil \frac{n^3}{13n-12} \right\rceil \;\leq\; \gamma(Q^3_n) \;\leq\;
  n\cdot\gamma(Q^2_n) \;\leq\;
  \frac{69}{133}\,n^2 + O(n),
\]
and Theorem~\ref{thm:projection} gives the additional lower bound
$\gamma(Q^3_n) \geq \gamma(Q^2_n)$. Hence $\gamma(Q^3_n) = \Theta(n^2)$ (however, this is asymptotically subsumed by Theorem ~\ref{thm:volume}).
The lower-bound constant is $1/13 \approx 0.077$ and the upper-bound constant
is $69/133 \approx 0.519$.

\section{Computational Certification}
\label{sec:computation}

\subsection{The domination ILP}

The domination number $\gamma(Q^3_n)$ is the optimal value of the integer linear
program
\begin{equation}
\label{eq:ilp}
  \min\bigl\{\mathbf{1}^\top x : (A+I)x \geq \mathbf{1},\;
  x \in \{0,1\}^{n^3}\bigr\},
\end{equation}
where $A$ is the $n^3 \times n^3$ adjacency matrix of $Q^3_n$. For each cell $v$
and each of the 13 canonical direction vectors $u$, all vertices on the maximal
queen line through $v$ in direction $u$ are adjacent to $v$; enumerating these
lines for all $v$ constructs the full adjacency relation in $O(n^4)$ time.
An implementation of this procedure, together with the solver configuration used
to obtain the values in Table~\ref{tab:values}, is available in \cite{code}.

\subsection{Automorphisms of $Q^3_n$ and the fundamental domain}

The octahedral group $O_h$ of order $48$, generated by coordinate permutations
and sign reflections, acts on $[n]^3$ by graph automorphisms of $Q^3_n$.
Corollary~\ref{cor:fundomain} guarantees that every dominating set has an
$O_h$-equivalent representative with a queen in the fundamental domain $F(n)$.
A canonical-orbit formulation can be encoded by adding, for each cell
$c \notin F(n)$, a linear inequality requiring that $c$ be selected only if some
lexicographically earlier cell is also selected. These constraints eliminate
symmetric candidates without excluding any optimal solution.

\subsection{Decomposition by first-queen position}

To certify that $k-1$ queens are insufficient, the feasibility space of
\eqref{eq:ilp} (with target $k-1$) is partitioned by the choice of the
lexicographically first queen. For each candidate first-queen cell $c$, a
subproblem is formed by fixing one queen at $c$ and forbidding all cells
lexicographically before $c$. The subproblems are mutually exclusive and
collectively exhaustive: their union covers every possible placement of a
lexicographically first queen without overlap. Every subproblem being infeasible
constitutes a certificate of optimality for the $k$-queen solution.

\subsection{Independent verification}

Every candidate solution is verified by an independent checker that confirms,
directly from the definition of $Q^3_n$ adjacency, that every cell of $[n]^3$
is dominated. This checker is independent of the constraint matrix used in
\eqref{eq:ilp}.

\subsection*{Results}

\begin{table}[ht]
\centering
\caption{Known values of $\gamma(Q^3_n)$.}
\label{tab:values}
\begin{tabular}{ccll}
\toprule
$n$ & $\gamma(Q^3_n)$ & Status & Example placement \\
\midrule
1 & 1 & Exact & $(0,0,0)$ \\
\midrule
2 & 1 & Exact & $(1,0,0)$ \\
\midrule
3 & 1 & Exact & $(1,1,1)$ \\
\midrule
4 & 4 & Exact & $(1,0,3),(1,1,0),(1,2,0),(1,3,3)$ \\
\midrule
5 & 6 & Exact & $(1,0,3),(1,1,0),(1,3,4),(1,4,1),(2,2,2),(3,2,2)$ \\
\midrule
6 & 8 & Exact & \makecell[l]{(2, 2, 2), (2, 2, 3), (2, 3, 2), (2, 3, 3), \\ (3, 2, 2), (3, 2, 3), (3, 3, 2), (3, 3, 3)} \\
\midrule
7 & 10--12 & Open & \makecell[l]{(0, 4, 3), (0, 6, 6), (1, 1, 5), (2, 3, 0), \\ (2, 4, 0), (3, 0, 2), (3, 6, 4), (4, 3, 6), \\ (4, 6, 4), (5, 0, 1), (6, 2, 3), (6, 5, 1)} \\
\bottomrule
\end{tabular}
\end{table}
\noindent The exact values for $n \leq 6$ are certified optimal using the Gurobi Optimizer \cite{gurobi}, with the computational lower and upper bounds of the ILP solver coinciding. For $n = 3$, the single centre cell dominates all 27 cells, as it lies on all 13 queen lines. For $n = 6$, the eight cells of the $2\times2\times2$ central block suffice, and the solver completed the proof of optimality in approximately 5 seconds. 

The value for $n = 7$ remains open. The worst-case search space heuristically scales as $O(2^{n^3})$; the solver was halted after 48 hours of computation for $n = 7$, establishing a computational lower bound of 10 and a best-found upper bound of 12.


\begin{thebibliography}{9}

\bibitem{barrrao2007}
J.~Barr and S.~Rao,
\textit{The $n$-queens problem in higher dimensions},
preprint, arXiv:0712.2309, 2007.
\url{https://doi.org/10.48550/arXiv.0712.2309}

\bibitem{cockayne1990}
E.~J.~Cockayne,
\textit{Chessboard domination problems},
\textit{Discrete Math.}\ \textbf{86} (1990), 13--20.

\bibitem{finozhenok2007}
D.~Finozhenok and W.~D.~Weakley,
\textit{An improved lower bound for domination numbers of the queen's graph},
\textit{Australas.\ J.\ Combin.}\ \textbf{37} (2007), 295--300.

\bibitem{osterCardWeakley}
P.~R.~J.~\"{O}sterg\r{a}rd and W.~D.~Weakley,
\textit{Values of domination numbers of the queen's graph},
\textit{Electron.\ J.\ Combin.}\ \textbf{8} (2001), no.~1, R29.

\bibitem{weakley2022}
W.~D.~Weakley,
\textit{Queen domination of even square boards},
\textit{Electron.\ J.\ Combin.}\ \textbf{29} (2022), no.~2, P2.50.

\bibitem{code}
\textit{3D queen domination: computational certification code},
Zenodo, 2026.
\url{https://doi.org/10.5281/zenodo.19412672}

\bibitem{gurobi}
Gurobi Optimization, LLC,
\textit{Gurobi Optimizer Reference Manual},
2026.
\url{https://www.gurobi.com}

\end{thebibliography}
\end{document}